\documentclass[portuges,12pt,letter]{article}
\usepackage[centertags]{amsmath}
\usepackage{amsfonts}
\usepackage{newlfont}
\usepackage{amscd}
\usepackage{graphics}
\usepackage{epsfig}
\usepackage{indentfirst}
\usepackage{amsxtra}
\usepackage[latin1]{inputenc}
\usepackage{amssymb, amsmath}
\usepackage{amsthm}
\usepackage[mathscr]{eucal}

\newtheorem{thm}{Theorem}[section]

\newtheorem{obe}[thm]{Remark}

\author{Fabio Silva Botelho\\ \\Department of Mathematics \\  Federal University of Santa Catarina, UFSC \\
Florian\'{o}polis, SC - Brazil}

\title{\bf  A generalization of the Ekeland variational principle}
\date{}
\begin{document}
\maketitle

\abstract{ In this short communication, we present a generalization of the Ekeland variational principle. The main result is established through
standard tools of functional analysis and calculus of variations. The novelty here is a result involving the second G\^{a}teaux variation of the functional in question.}
 %In a second step we develop a new matrix version of the generalized method of lines, applicable to a large class of related models.}

\section{Introduction}

In this article we present and prove a generalization of  the Ekeland variational principle. A proof of the so far known principle may be found in Giusti, \cite{906}, pages 160-161. With slight improvements a similar result is presented in \cite{120}.

We also highlight details on the function spaces addressed may be found in \cite{1}.

 At this point we state such a result.

\begin{thm}[Ekeland variational principle] Let $(U,d)$ be a complete metric space and let $F:U \rightarrow \overline{\mathbb{R}}\equiv \mathbb{R} \cup\{+\infty\}$ be a lower semi continuous
bounded below functional taking a finite value at some point.

Let $\varepsilon>0$.  Assume for some $u \in U$ we have
$$F(u) \leq \inf_{u \in U}\{F(u)\}+ \varepsilon.$$
Under such hypotheses, there exists $v \in U$ such that
\begin{enumerate}
\item\label{p11} $d(u,v) \leq 1$,
\item\label{p12} $F(v) \leq F(u)$,
\item\label{p13} $F(v) \leq F(w)+ \varepsilon d(v,w),\; \forall w \in U$.
\end{enumerate}
\end{thm}

\section{The generalized Ekeland variational principle}

In this section we state and prove the following new result, which the proof is based on the one presented in \cite{906}.
\begin{thm}[Generalized Ekeland variational principle] Let $(U,d)$ be a complete metric space and let $F:U \rightarrow \overline{\mathbb{R}}$ be a lower semi continuous
bounded below functional taking a finite value at some point.

Let $\varepsilon>0$.  Assume for some $u \in U$ we have
$$F(u) \leq \inf_{u \in U}\{F(u)\}+ \varepsilon.$$
Under such hypotheses, there exists $v \in U$ such that
\begin{enumerate}
\item\label{p11} $d(u,v) \leq 1$,
\item\label{p12} $F(v) \leq F(u)$,
\item\label{p13} $F(v) \leq F(w)+ \varepsilon d(v,w),\; \forall w \in U$.
\item\label{p14} Assuming $U$ is a Banach space and $F$ is G\^{a}teaux differentiable, we have
$$\|\delta F(v)\|_{U^*} \leq \varepsilon.$$
\item\label{p15} Finally, assuming also $F$ is twice Fr\'{e}chet differentiable, we have
$$\delta^2F(v,\varphi, \varphi) \geq -4\;\varepsilon \|\varphi\|_U-2\frac{o(\varepsilon^2)}{\varepsilon^2},\; \forall \varphi \in U,$$
where $$\frac{o(\varepsilon^2)}{\varepsilon^2} \rightarrow 0, \text{ as } \varepsilon \rightarrow 0^+.$$
\end{enumerate}
\end{thm}
\begin{proof} Define the sequence $\{u_n\} \subset U$ by:
$$u_1=u,$$
and having $u_1,...,u_n$, select $u_{n+1}$ as specified in the next lines. First, define
$$S_n=\{w \in U \;|\; F(w) \leq F(u_n)-\varepsilon d(u_n,w)\}.$$

Observe that $u_n \in S_n$ so that $S_n$ in non-empty.

On the other hand, from the definition of infimum, we may select $u_{n+1} \in S_n$ such that
\begin{equation}\label{sp51}F(u_{n+1}) \leq \frac{1}{2}\left\{F(u_n)+\inf_{w \in S_n}\{F(w)\}\right\}.\end{equation}

Since $u_{n+1} \in S_n$ we have
\begin{equation}\label{sp48} \varepsilon d(u_{n+1},u_n) \leq F(u_n)-F(u_{n+1}).\end{equation}
and hence
\begin{equation}\label{sp49} \varepsilon d(u_{n+m},u_n) \leq  \sum_{i=1}^m \varepsilon d(u_{n+i},u_{n+i-1}) \leq F(u_n)-F(u_{n+m}).\end{equation}

From (\ref{sp48}), $\{F(u_n)\}$ is a decreasing sequence bounded below by $\inf_{u \in U}F(u)$ so that there exists $ \alpha \in \mathbb{R}$ such that
$$F(u_n) \rightarrow \alpha \text{ as } n \rightarrow \infty.$$

From this and (\ref{sp49}), $\{u_n\}$ is a Cauchy sequence , converging to some $v \in U.$

Since $F$ is lower semi-continuous we get,
$$\alpha =\liminf_{m \rightarrow \infty}F(u_{n+m}) \geq F(v),$$
so that letting $m \rightarrow \infty$ in (\ref{sp49}) we obtain
\begin{equation}\label{sp50}\varepsilon d(u_n,v) \leq F(u_n)-F(v),\end{equation}
and, in particular for $n=1$ we get
$$0 \leq  \varepsilon d(u,v) \leq F(u)-F(v)\leq F(u)-\inf_{u \in U}F(u) \leq \varepsilon.$$
Thus, we have proven \ref{p11} and \ref{p12}.

Suppose, to obtain contradiction, that \ref{p13} does not hold.

Hence, there exists $w \in U$ such that $$F(w)< F(v)- \varepsilon d(w,v).$$

In particular we have \begin{equation}\label{pt} w \neq v.\end{equation}
Thus, from this and (\ref{sp50}) we have
$$F(w)< F(u_n)-\varepsilon(u_n,v)-\varepsilon d(w,v) \leq F(u_n)-\varepsilon d(u_n,w), \forall n \in \mathbb{N}.$$

Now observe that $w \in S_n, \forall n \in \mathbb{N}$ so that $$\inf_{w \in S_n}\{F(w)\} \leq F(w), \forall n \in \mathbb{N}.$$

From this and (\ref{sp51}) we obtain,

$$2F(u_{n+1})-F(u_n) \leq F(w)< F(v)- \varepsilon d(v,w),$$
so that
$$2\liminf_{n \rightarrow \infty}\{F(u_{n+1})\} \leq F(v)- \varepsilon d(v,w) +\liminf_{ n \rightarrow \infty}\{F(u_n)\}.$$
Hence,
$$F(v) \leq \liminf_{n \rightarrow \infty}\{F(u_{n+1})\} \leq F(v)- \varepsilon d(v,w),$$
so that
$$0 \leq- \varepsilon d(v,w),$$
which contradicts (\ref{pt}).

Thus \ref{p13} holds.

%\end{proof}
Assume now $U$ is a Banach space, $F$ is G\^{a}teaux differentiable and $\varphi \in U$. Fix $t \in (0,1)$.

Thus, from \ref{p13},
\begin{gather}
F(v)-F(v+t\varphi)\leq \varepsilon \|t\varphi\|_U,
\end{gather}
so that
\begin{gather}
\frac{F(v)-F(v+t\varphi)}{t}\leq \varepsilon \|\varphi\|_U,
\end{gather}

Therefore, letting $t \rightarrow 0^+$, we get
\begin{gather}
 -\langle \delta F(v),\varphi \rangle_U \leq \varepsilon \|\varphi\|_U.
\end{gather}
Similarly, for $t \in (0,1)$,
\begin{gather}
F(v)-F(v+t(-\varphi))\leq  \varepsilon\| t \varphi\|_U,
\end{gather}
so that,
 \begin{gather}
\frac{F(v)-F(v+t(-\varphi))}{t}\leq \varepsilon \|\varphi\|_U.
\end{gather}
Letting $t \rightarrow 0^+$, we obtain
\begin{gather}
 \langle \delta F(v),\varphi\rangle_U \leq \varepsilon \|\varphi\|_U,
\end{gather}
so that
\begin{gather}
 |\langle \delta F(v),\varphi\rangle_U |=\varepsilon \|\varphi\|_U,\; \forall \varphi \in U.
\end{gather}
Thus,
\begin{gather}
\|\delta F(v)\|_{U^*}\leq \varepsilon.
\end{gather}

Assume here, in addition, $F$ is twice Fr\'{e}chet differentiable in $U$. From \ref{p13}, with $\varepsilon^2$ replacing $\varepsilon$ in the previous items, we have
$$F(v+\varepsilon \varphi)-F(v) \geq -\varepsilon^2 \|\varepsilon \varphi\|_U,$$
so that
 from this and the twice Fr\'{e}chet differentiability hypothesis, we get
$$\varepsilon \langle \delta F(v), \varphi\rangle_U +\frac{1}{2}\varepsilon^2\delta^2F(v,\varphi,\varphi) +o(\varepsilon^2)\geq -\;\varepsilon^3 \| \varphi\|_U,$$
so that, from this and $$|\langle \delta F(v), \varphi\rangle_U| \leq \varepsilon^2 \|\varphi\|_U,$$ we obtain
\begin{eqnarray}\frac{1}{2}\delta^2F(v,\varphi,\varphi)  &\geq& -\;\varepsilon \| \varphi\|_U-\varepsilon\frac{|\langle \delta F(v), \varphi\rangle_U|}{\varepsilon^2}-\frac{o(\varepsilon^2)}{\varepsilon^2}
\nonumber \\ &\geq&-2\;\varepsilon \| \varphi\|_U-\frac{o(\varepsilon^2)}{\varepsilon^2}.\end{eqnarray}

Hence,
$$\delta^2F(v,\varphi, \varphi) \geq -4\;\varepsilon \|\varphi\|_U-2\frac{o(\varepsilon^2)}{\varepsilon^2},\; \forall \varphi \in U,$$
where
$$\frac{o(\varepsilon^2)}{\varepsilon^2} \rightarrow 0, \text{ as } \varepsilon \rightarrow 0^+.$$
The proof is complete.
\end{proof}
\begin{obe} We may introduce in $U$ a new metric given by $d_1=\varepsilon^{1/2}d.$
We highlight that the topology remains the same and also $F$ remains lower semi-continuous.
Under the hypotheses of the last theorem, for a not relabeled metric $d$,  if  $u \in U$ is such that $F(u) < \inf_{u \in U}F(u)+\varepsilon^2,$ then
there exists $v \in U$ such that
\begin{enumerate}
\item\label{p11.A} $d(u,v) \leq \varepsilon^{1/2}$,
\item\label{p12.A} $F(v) \leq F(u)$,
\item\label{p13.A} $F(v) \leq F(w)+ \varepsilon^{3/2} d(v,w),\; \forall w \in U$.
\item\label{p14.A} Assuming $U$ is a Banach space and $F$ is G\^{a}teaux differentiable, we have
$$\|\delta F(v)\|_{U^*} \leq \varepsilon^{3/2}.$$
\item\label{p15.A} Finally, assuming also $F$ is twice Fr\'{e}chet differentiable, we have
$$\delta^2F(v,\varphi, \varphi) \geq -4\;\varepsilon^{1/2} \|\varphi\|_U-2\frac{o(\varepsilon^2)}{\varepsilon^2},\; \forall \varphi \in U,$$
where $$\frac{o(\varepsilon^2)}{\varepsilon^2} \rightarrow 0, \text{ as } \varepsilon \rightarrow 0^+.$$
\end{enumerate}
\end{obe}

\end{document}